\documentclass[12pt]{amsart}

\usepackage{amssymb,amsmath,amsthm,enumerate}

\newcommand{\NN}{\mathbb{N}}
\newcommand{\OO}{\mathcal{O}}
\newcommand{\PP}{\mathbb{P}}
\newcommand{\RR}{\mathbb{R}}

\DeclareMathOperator{\Bs}{Bs}
\DeclareMathOperator{\codim}{codim}
\DeclareMathOperator{\Pic}{Pic}

\theoremstyle{plain}
\newtheorem{theorem}{Theorem}
\newtheorem{lemma}[theorem]{Lemma}
\newtheorem{corollary}[theorem]{Corollary}
\newtheorem{proposition}[theorem]{Proposition}

\begin{document}

\mbox{}
\vspace{-1.1ex}
\title[Fano varieties with finitely generated Okounkov semigroups]{Fano varieties with finitely generated semigroups in the Okounkov body construction}
\author{Shin-Yao Jow}
\address{Department of Mathematics,
National Tsing Hua University,
Taiwan}
\email{\texttt{syjow@math.nthu.edu.tw}}
\date{}

\begin{abstract}
The Okounkov body is a construction which, to an effective divisor $D$ on an $n$-dimensional algebraic variety $X$, associates a convex body $\Delta(D)$ in the $n$-dimensional Euclidean space $\RR^n$. It may be seen as a generalization of the moment polytope of an ample divisor on a toric variety, and it encodes rich numerical information about the divisor $D$. When constructing the Okounkov body, an intermediate product is a lattice semigroup $\Gamma(D)\subset \NN^{n+1}$, which we will call the Okounkov semigroup. Recently it was discovered that finite generation of the Okounkov semigroup has interesting geometric implication for $X$ regarding toric degenerations and integrable systems, however the finite generation condition is difficult to establish except for some special varieties $X$. In this article, we show that smooth projective Fano varieties of coindex${}\le 2$ have finitely generated Okounkov semigroups, providing the first family of nontrivial higher dimensional examples that are not coming from representation theory. Our result also gives a partial answer to a question of Anderson, K\"uronya, and Lozovanu.
\end{abstract}


\maketitle

\section{Introduction}

A few years ago, Kaveh-Khovanskii \cite{KK} and Lazarsfeld-Musta\c{t}\v{a} \cite{LM} independently discovered very similar constructions which, to an effective (Cartier) divisor $D$ (or more generally a graded linear system) on an $n$-dimensional algebraic variety $X$, associates a convex body $\Delta(D)$ in the $n$-dimensional Euclidean space $\RR^n$. The convex body $\Delta(D)$, called the Okounkov body in \cite{LM}, may be seen as a generalization of the moment polytope of an ample divisor on a toric variety, and it encodes rich numerical information about the divisor $D$ \cite{J}. The construction of $\Delta(D)$ in \cite{LM} can be outlined in three steps as follows.
\begin{description}
\item[Step~1] Fix a flag of subvarieties $Y_\bullet$ on $X$
\[
 Y_\bullet \colon X=Y_0\supset Y_1 \supset Y_2 \supset \cdots \supset Y_{n-1} \supset Y_n
 =\{\text{pt}\}  \]
such that each $Y_i$ is smooth at the point $Y_n$. For any effective divisor $E$ on $X$, the flag determines a valuation-like function $\nu(E)=\bigl(\nu_1(E),\ldots,\nu_n(E)\bigr)\in \NN^n$ in the following manner: $\nu_1(E)$ is the vanishing order of $E$ along $Y_1$, $\nu_2(E)$ is the vanishing order along $Y_2$ of $E-\nu_1(E)Y_1$ restricted to $Y_1$, and so on.
\item[Step~2] Let $\Gamma(D)=\Gamma_{Y_\bullet}(D)=\{(\nu(E),m)\mid E\in |mD|, m\in \NN\}$. This is a sub-semigroup of $\NN^n\times\NN$, which we will call the \emph{Okounkov semigroup} of $D$ with respect to $Y_\bullet$.
\item[Step~3] The Okounkov body of $D$ with respect to $Y_\bullet$ is \[
  \Delta(D)=\Delta_{Y_\bullet}(D)=\overline{\mathrm{cone}(\Gamma(D))} \cap (\RR^n\times \{1\}), \]
  where $\overline{\mathrm{cone}(\Gamma(D))}\subset \RR^n\times \RR$ is the closed convex cone spanned by $\Gamma(D)$.
\end{description}

A similar construction can be applied to any graded linear system $V_\bullet$ on $X$. A graded linear system $V_\bullet$ on $X$ associated with a divisor~$D$ is a collection of linear systems $\{V_m\subseteq |mD| \bigm| m\in\NN\}$ such that $V_i+V_j\subseteq V_{i+j}$ for all $i,j\in\NN$ (where the sum denotes the sum of divisors in $V_i$ and $V_j$). The condition $V_i+V_j\subseteq V_{i+j}$ and the valuation-like property of $\nu$ imply that the Okounkov semigroup $\Gamma(V_\bullet)$ can still be defined by \[ \Gamma(V_\bullet)=\{(\nu(E),m)\mid E\in V_m, m\in \NN\}, \]
which in turn gives the Okounkov body $\Delta(V_\bullet)=\overline{\mathrm{cone}(\Gamma(V_\bullet))} \cap (\RR^n\times \{1\})$.

Although the Okounkov body may not be rational polyhedral in general, it is clear from the definition that if the Okounkov semigroup is finitely generated, then the Okounkov body is rational polyhedral. Anderson \cite{A} showed that if $D$ is a very ample divisor on a projective variety~$X$, and if the graded linear system $V_\bullet$ given by \[
 V_m=m|D|=\{D_1+\cdots+D_m\mid D_1,\ldots,D_m\in |D|\} \]
has finitely generated Okounkov semigroup, then $X$ admits a flat degeneration to a (not necessarily normal) toric variety whose normalization is the toric variety associated to the rational polytope $\Delta(V_\bullet)$. Furthermore, Harada and Kaveh \cite{HK} showed that this toric degeneration can be used to construct a complete integrable system on $X$ with moment polytope $\Delta(V_\bullet)$. Unfortunately, finite generation of the Okounkov semigroup is a condition which is difficult to establish except for a few special cases, and the examples given in \cite{A} and \cite{HK} are either representation-theoretic or of dimension less than three. The aim of this article is to provide another family of higher-dimensional examples which satisfy that elusive yet important condition:

\begin{theorem}  \label{t:main}
Let $X$ be a smooth complex projective Fano variety of dimension~$n$ and index~$r$, and let $H$ be an ample divisor on $X$ such that $-K_X=rH$. Let $c$ be a positive integer such that $D=cH$ is very ample, and let $V_\bullet$ be the graded linear system given by $V_m=m|D|$ for all $m\in\NN$. If\/ $r\ge n-1$, then there exists a flag \[
 Y_\bullet \colon X=Y_0\supset Y_1 \supset Y_2 \supset \cdots \supset Y_{n-1} \supset Y_n
 =\{\text{pt}\}  \]
of irreducible smooth subvarieties of $X$ such that
\begin{enumerate}[\upshape (i)]
 \item The Okounkov semigroups $\Gamma_{Y_\bullet}(V_\bullet)$ and $\Gamma_{Y_\bullet}(D)$ are both finitely generated.
 \item The Okounkov bodies $\Delta_{Y_\bullet}(V_\bullet)$ and $\Delta_{Y_\bullet}(D)$ are both equal to the simplex in $\RR^n$ with vertices $0,ce_1,\ldots,ce_{n-1}$ and $cde_n$, where $d=H^n$ and $\{e_1,\ldots,e_n\}$ is the standard basis for~$\RR^n$.
\end{enumerate}
\end{theorem}

As discussed above, this theorem implies the following
\begin{corollary}
 A smooth complex projective Fano variety $X$ of coindex${}\le 2$ admits a flat degeneration to a (not necessarily normal) toric variety whose normalization is a weighted projective space, and there exists a complete integrable system on $X$ whose moment polytope is a simplex.
\end{corollary}

It is well-known that the dimension~$n$ and the index~$r$ of a Fano manifold $X$ satisfy $r\le n+1$ \cite[Corollary~2.1.13]{PS}, and that if $r=n+1$ then $X$ is the projective space, and if $r=n$ then $X$ is a quadric \cite[Theorem~3.1.14]{PS}. If $r=n-1$ then $X$ is often called a del Pezzo variety, which has also been completely classified \cite[Theorem~3.3.1]{PS}, but our proof of Theorem~\ref{t:main} is not by individually verifying each case on the classification list.

Our proof of Theorem~\ref{t:main} actually shows that $X$ satisfies a geometric property which is stronger than finite generation of the Okounkov semigroup. Roughly speaking, the property is that there exist $H_1,\ldots,H_n\in |H|$ which intersect at only one point set-theoretically. See Proposition~\ref{p:flag} for a more precise statement. A slight variant of this property and its implication for finite generation of the Okounkov semigroup were first noted in \cite{AKL}, and finding varieties with this property was posed as an open question at the end of \emph{loc. cit.} Thus our result also provides del Pezzo varieties as the first nontrivial answer to that question, as far as we know.

\section{Proof of Theorem~\ref{t:main}}

\begin{lemma}  \label{l:EC}
Every complete linear system on a smooth elliptic curve contains a member whose support is a single point.
\end{lemma}

\begin{proof}
Let $C$ be a smooth elliptic curve, and let $\Pic^d(C)$ be the Picard variety parametrizing line bundles of degree~$d$ on $C$, where $d$ is any positive integer. The morphism $C\to \Pic^d(C)$ given by $P\mapsto dP$ is finite and surjective, hence every effective divisor of degree~$d$ on $C$ is linearly equivalent to $dP$ for some $P\in C$.
\end{proof}

\begin{proposition}  \label{p:flag}
Let $X$ be a smooth complex projective Fano variety of dimension~$n$ and index~$r$, and let $H$ be an ample divisor on $X$ such that $-K_X=rH$. If\/ $r\ge n-1$, then there exists a flag
\[
 Y_\bullet \colon X=Y_0\supset Y_1 \supset Y_2 \supset \cdots \supset Y_{n-1} \supset Y_n
 =\{\text{pt}\}  \]
of irreducible smooth subvarieties of $X$, where $\codim_X(Y_i)=i$, which satisfies both of the following conditions:
 \begin{enumerate}[\upshape (i)]
  \item There exist $n-1$ general members $H_1,\ldots,H_{n-1}\in |H|$ of the linear system $|H|$ such that $Y_i$ is the scheme-theoretic intersection of $H_1,\ldots,H_i$ for all $i\le n-1$.
  \item There exists a member of the linear system $|H|$ whose set-theoretic intersection with $Y_{n-1}$ is the single point $Y_n$.
 \end{enumerate}
\end{proposition}

\begin{proof}
If $r=n+1$, then $X=\PP^n$ and $H$ is a hyperplane, so the lemma obviously holds in this case. If $r=n$, then $X$ is a quadric in $\PP^{n+1}$ and $H$ is a hyperplane section. Let $H_1,\ldots,H_{n-1}\in |H|$ be general members of the linear system $|H|$, and let $Y_i$ be the scheme-theoretic intersection of $H_1,\ldots,H_i$ for each $i\le n-1$. Then $Y_1,\ldots,Y_{n-1}$ are smooth and irreducible by Bertini's theorem. The curve $Y_{n-1}$ is a conic in this case, so the point $Y_n$ can be chosen arbitrarily on $Y_{n-1}$, and any hyperplane tangent to $Y_{n-1}$ at $Y_n$ gives a divisor which satisfies (ii).

The remaining case is $r=n-1$, that is, $X$ is a del Pezzo variety. By \cite[Proposition~3.2.4~(i)]{PS}, the base locus $\Bs|H|$ is empty if $H^n>1$, while if $H^n=1$ then $\Bs|H|$ consists of a single point at which a general member of $|H|$ is smooth. This together with Bertini's theorem imply that a general member of $|H|$ is smooth. Pick any smooth $H_1\in |H|$ and set $Y_1=H_1$. By \cite[Proposition~3.2.3]{PS}, $(Y_1,H|_{Y_1})$ is a del Pezzo variety as well, so a general member of $|H|_{Y_1}|$ is again smooth. Pick any smooth $Y_2\in |H|_{Y_1}|$. Since $H^0\bigl(X,\OO_X(H)\bigr)\to H^0\bigl(Y_1,\OO_{Y_1}(H|_{Y_1})\bigr)$ is surjective \cite[Proposition~3.2.3]{PS}, there exists an $H_2\in |H|$ such that $Y_2$ is the (scheme-theoretic) intersection of $H_1$ and $H_2$. Continue in this way to inductively select the subvarieties $Y_i$ and the divisors $H_i\in |H|$ for $i=1,2,\ldots, n-1$. Then the curve $Y_{n-1}$ is a smooth elliptic curve, and the restriction of $|H|$ on $Y_{n-1}$ is a complete linear system. Hence by Lemma~\ref{l:EC}, there exists an $H_n\in |H|$ whose set-theoretic intersection with $Y_{n-1}$ is a single point, and we set $Y_n$ to be this point.
\end{proof}

\begin{proof}[Proof of Theorem~\ref{t:main}]
Let $Y_\bullet$ be a flag of smooth subvarieties of $X$ as in Proposition~\ref{p:flag}. The proof of \cite[Proposition~4]{AKL} goes without change to show that the Okounkov body $\Delta_{Y_\bullet}(H)$ is the $n$-dimensional simplex $\Delta\subset \RR^n$ with vertices $0,e_1,\ldots,e_{n-1}$ and $de_n$. By the homogeneity property of Okounkov bodies \cite[Proposition~4.1]{LM},  \[
     \Delta_{Y_\bullet}(D)=\Delta_{Y_\bullet}(cH)=c\,\Delta. \]
Obviously $\Delta_{Y_\bullet}(V_\bullet)\subseteq \Delta_{Y_\bullet}(D)=c\,\Delta$ since $V_m\subseteq |mD|$ for all $m\in\NN$. On the other hand, by Proposition~\ref{p:flag}, the set of valuation vectors $\nu_{Y_\bullet}(|H|)$ contains every vertex of the simplex $\Delta$, hence the set of valuation vectors $\nu_{Y_\bullet}(V_1)=\nu_{Y_\bullet}(|D|)$ contains every vertex of the simplex $c\,\Delta$. This implies that $\Delta_{Y_\bullet}(V_\bullet)=c\,\Delta$, and then by \cite[Proposition~1]{A} the Okounkov semigroups $\Gamma_{Y_\bullet}(V_\bullet)$ and $\Gamma_{Y_\bullet}(D)$ are both finitely generated.
\end{proof}

\noindent\textbf{Acknowledgments.} The author gratefully acknowledges the support of MOST (Ministry of Science and Technology, Taiwan), NCTS (National Center for Theoretical Sciences), and the Kenda Foundation.



\begin{thebibliography}{4}

\bibitem{A}
D.~Anderson, \emph{Okounkov bodies and toric degenerations}, Math. Ann. \textbf{356}~(2013), 1183--1202.

\bibitem{AKL}
D.~Anderson, A.~K\"uronya, and V.~Lozovanu, \emph{Okounkov bodies of finitely generated divisors}, Int. Math. Res. Notices Vol.~2014, No.~9, 2343--2355.

\bibitem{HK}
 M.~Harada and K.~Kaveh,
 \emph{Integrable systems, toric degenerations and Okounkov bodies}, preprint, \texttt{arXiv:1205.5249}

\bibitem{J}
S.-Y. Jow, \emph{Okounkov bodies and restricted volumes along very general curves}, Adv. Math. \textbf{223}~(2010), 1356--1371.

\bibitem{KK}
K.~Kaveh and A.~Khovanskii, \emph{Newton-Okounkov bodies, semigroups of integral points, graded
algebras and intersection theory}, Ann. Math. \textbf{176}~(2012), 925--978.

\bibitem{LM}
R.~Lazarsfeld and M.~Musta\c{t}\v{a}, \emph{Convex bodies associated to linear
series}, Ann. Sci. Ecole Norm. Sup. \textbf{42}~(2009), 783--835.

\bibitem{PS}
A.N.~Parshin and I.R.~Shafarevich (eds.), \emph{Algebraic Geometry V: Fano Varieties}, Encyclopaedia of Mathematical Sciences vol.~47, Berlin: Springer, 1999.

\end{thebibliography}
\end{document}